\documentclass[a4paper,10pt]{amsart}

\usepackage[english]{babel}
\usepackage{dsfont} % \mathds

\usepackage{amsmath}
\usepackage{amsfonts}
\usepackage{amssymb} % \upharpoonright
\usepackage{amsthm} % \newtheorem
\usepackage{color}

\newtheorem{thm}{Theorem}
\newtheorem{dfn}[thm]{Definition}
\newtheorem{lem}[thm]{Lemma}

\newtheorem{exa}[thm]{Example}
\newtheorem{prop}[thm]{Proposition}
\newtheorem{rem}[thm]{Remark}

\newcommand{\hide}[1]{}

\newcommand{\C}{\mathds{C}}
\newcommand{\R}{\mathds{R}}
\newcommand{\N}{\mathds{N}}

\newcommand{\cB}{\mathcal{B}}
\newcommand{\cC}{\mathcal{C}}
\newcommand{\cF}{\mathcal{F}}

\newcommand{\cP}{\mathcal{P}}

\newcommand{\cD}{\mathcal{D}}
\newcommand{\cK}{\mathcal{K}}

\newcommand{\cG}{\mathcal{G}}
\newcommand{\cS}{\mathcal{S}}
\newcommand{\cT}{\mathcal{T}}

\newcommand{\ii}{\mathsf{i}}

\newcommand{\cI}{\mathcal{I}}

\author{Konrad Schm\"udgen}
\address{Universit\"at Leipzig, Mathematisches Institut, Augustusplatz 10/11, D-04109 Leipzig, Germany}
\email{schmuedgen@math.uni-leipzig.de}

\begin{title}
{On the infinite-dimensional  moment problem}
\end{title}
\date{}
\begin{document}

\begin{abstract}
This paper deals with  the  moment problem on a (not necessarily finitely generated) commutative unital real algebra $A$. We define  moment functionals  on $A$ as linear functionals which can be written as  integrals over characters of $A$ with respect to cylinder measures. Our main results provide such integral representations for $A_+$--positive linear functionals (generalized Haviland theorem) and for positive functionals fulfilling Carleman conditions. As an application we solve  the moment problem for the symmetric algebra $S(V)$ of a real vector space $V$. As a byproduct we obtain  a new approaches to 
the moment problem on $S(V)$ for a nuclear  space $V$ and to the integral decomposition of continuous positive functionals on a barrelled nuclear topological algebra $A$.
\end{abstract}

\maketitle

\textbf{AMS  Subject  Classification (2010)}.
 44 A 60 (Primary), 46 G 12, 28 C 20 (Secondary).\\

\textbf{Key  words:}  moment problem, cylinder measure, symmetric algebra, nuclear space, Carleman condition

\section{Introduction}%%%
The moment problem on $\R^d$ is a well-studied  classical research
topic, see e.g. \cite{Ak}, \cite{berg84}, \cite{schm2017}. In modern formulation,  it is the following question: When can a linear functional  on the polynomial algebra $\R[x_1,\dots,x_d]$ (or on a finitely generated commutative real algebra $A$) be expressed as an integral over point evaluations with respect to some Radon measure on $\R^d$ (or on the  character space $\hat{A}$)?

This paper is about the moment problem in infinitely many variables. Suppose that $A$ is a (not necessarily finitely generated)  commutative unital real algebra. We study the following generalization of the classical moment problem: 

{\it When is a linear functional $L$ on $A$ an integral over characters of $A$? }

\noindent
If $A$ is finitely generated,  Haviland's theorem \cite[Theorem 1.14]{schm2017} provides an answer: it holds if and only if $L$ is $A_+$-positive, that is, $L$ is nonnegative on  elements of $A$ that are
nonnegative on all characters. Note that the latter condition is difficult to verify in general. Further, if  $L$ is only  nonnegative on squares and  satisfies  Carleman growth conditions, then it is a moment functional \cite[Theorem ]{schm2017}. The aim of the present paper is to generalize these two fundamental  results to   not necessarily finitely generated  algebras.  

From the conceptional side, our  new contribution  is that we do not define moment functionals by integrals with respect to  Radon measures, but only as integrals over cylinder measures.  This allows us to solve moment problems for more general algebras without nuclearity assumptions. 

Our main result in this respect   states that under corresponding  natural assumptions a linear functional $L$ on $A$ is an integral of characters over a cylinder measure. 

The next natural question is: When is $L$ an integral of some ($\sigma$-additive) measure?  From Minlos' theorem it follows that this holds if some   nuclearity assumption is added. As an immediate consequence  we obtain new and unifying approaches for two fundamental  results: the moment problem for  symmetric algebras over  nuclear spaces (Theorem \ref{mpsymmetrc}) and the  decomposition of continuous positive  functionals on barrelled nuclear topological algebras as integrals over characters (Theorem \ref{intcharalgebra}).
Note that for both results the known proofs contained in the literature are quite involved and based on different techniques.

This paper is organized as follows. Section \ref{cylindersec}  reviews basic definitions and results on cylinder measures. In Section \ref{momentfunccylsec} we  define moment functionals with respect to cylinder measures.  Section \ref{haviland} contains the main results of this paper. We prove an  integral representation theorem (Theorem \ref{mpscylinder}) 
with respect to cylinder measures for $A_+$-positive functionals and for positive functionals satisfying  Carleman conditions. In Section \ref{nucasssec} we assume that the underlying space is nuclear. Then, combined with Minlos's theorem, we obtain integral representations  by  "ordinary" $\sigma$-additive measures (Theorems \ref{nucearass} and \ref{intcharalgebra}). 
In Section \ref{mpsymmetricalgebras} we treat the moment problem for the symmetric algebra $S(V)$ of a real vector space $V$ (Theorem \ref{mpsymmetrc}).
 Section \ref{Gausssec} deals with moment functionals with respect to Gaussian cylinder measures. We give examples which have  representing measures and examples for which there are only representing cylinder measures, but  no representing measures.
 
 All results in this paper are developed for a {\it real} commutative algebra $A$, but they are easily extended to {\it complex} commutative $*$-algebras by passing to the  complexification $A_\C$ of $A.$  
Recall that the direct sum $A_\C:=A\oplus {\ii} A$ of vector  spaces $A$ and ${\ii} A$ becomes a complex commutative    $*$-algebra, called the   complexification  of $A$, with multiplication, involution  and  scalar multiplication defined by 
\begin{align*}
(a+{\ii} b)(c+{\ii}d)=ac-bd+ &{\ii}(bc+ad),~~ (a+{\ii} b)^*:=a-{\ii} b,\\ (\alpha+ {\ii}\beta)(a+ {\ii} b)&:= \alpha a-\beta b + {\ii} (\alpha b +\beta a),
\end{align*}
where $a,b,c,d\in A$ and $\alpha,\beta \in \R$. Each $\R$-linear functional $L:A\to \R$ extends uniquely to a $\C$-linear functional $L_\C:A_\C\to \C$ by $L_\C(a+{\ii}b)=L(a)+{\ii}L(b)$, $a,b\in A$. Clearly,\,  $L(a^2)\geq 0$ for all $a\in A$ if and only if\, $L_\C(x^*x)\geq 0$ for all $x\in A_\C$. 

Infinite-dimensional moment problems were first studied by A. G. Kostyuchenko and B.S. Mityagin \cite{km}, 
Y. M. Berezansky and S.N. Sifrin \cite{bs},  
  H.J. Borchers and J. Yngvason \cite{by} for symmetric algebras over nuclear spaces and since then  by various authors, see   \cite{sifrin}, \cite{cs}, \cite{by2}, \cite{bk}, \cite{heg}, \cite{Sch1990}, \cite{ah}, \cite{gik}, \cite{ikr}, \cite{aljorkim}.  
 Infinite-dimensional versions  of the moment theorem for compact semi-algebraic sets  \cite{Sch1991} are 
 obtained 
 in  \cite{ah}, \cite{aljorkim}.

\section{Cylinder measures}\label{cylindersec} 

In this section we recall some basics on cylinder measures; our standard reference is \cite[Chapter IV]{gw4}. 

Suppose that $V$ is  a real locally convex Hausdorff space. Let $V'$ denote the vector space of all continuous real-valued linear functionals on $V$.

Suppose that $\cF$ is a {\it finite-dimensional} linear subspace  of $V$. Its annihilator space $\cF^\circ$ is defined by 
 $$
 \cF^\circ=\{ L\in V': L(f)=0 ~~{\rm for}~f\in \cF\}.
 $$ 
The quotient space $V/\cF^\circ$ is also finite-dimensional. If $M$ is a subset of $V/\cF^\circ$, then the set $N$ of all $L\in V'$ which are mapped into elements of $M$ under the canonical mapping of $V'$ into $V'/ \cF^\circ$ is called the {\it cylinder set with base $M$ and generating subspace $V'/ \cF^\circ$}. The cylinder sets form an  algebra  of sets denoted by $\cC(V)$. Let $\cB(V)$ denote the $\sigma$-algebra generated by all cylinder sets. Clearly, $\cB(V)$ is the Borel $\sigma$-algebra when the vector space $V$ is equipped with the weak topology $\sigma(V,V')$.

 From \cite[Chapter IV, \S 1,3.]{gw4}  we restate the following  definition.
 \begin{dfn}\label{cmd}
 A \emph{cylinder measure} on $V'$ is a real-valued function $\mu$ on the algebra $\cC(V)$ such that:
 \begin{itemize}
 \item[(i)] $0\leq \mu(N)\leq 1$ for all $N\in \cC(V)$.
 \item[(ii)]   $\mu(V')=1$.
 \item[(iii)] If $N$ is the union of pairwise disjoint  sets $N_n\in \cC(V)$, $ n\in \N$, with Borel base sets and a common generating subspace $V'/ \cF^\circ$, then
 $$
 \mu(N)=\sum_{n=1}^\infty \mu(N_n).
 $$
 \end{itemize} 
 \end{dfn}
 In \cite{gw4} there is a fourth condition, but it follows from the three others.

\begin{dfn}\label{defconti} 
A cylinder measure $\mu$ on $V'$ is called \emph{continuous} if for each $\varepsilon >0$ and $a>0$ there exists a zero neighbourhood $U$  in $V$ such that for all $t\in U$ we have
\begin{align}
\mu\big( \{ \chi\in V': |\chi(t)|\geq a\}\big)\leq \varepsilon.
\end{align}
\end{dfn}
This  definition of continuity is equivalent to the definition given in \cite{gw4}, as noted in \cite[Ch. IV, \S 1, 4.]{gw4} and proved in \cite{minlos}. It is easily seen that we   may restrict ourselves to the case when $a=1$.

Now suppose that $\mu$ is a cylinder measure on $V'$. Then $\mu$ defines a normalized measure $\nu_\cF$  on the Borel $\sigma$-algebra of  each factor space $V'/ \cF^\circ$  by 
\begin{align}\label{nufmun}
 \nu_\cF(M):=\mu(N),
 \end{align} 
 where $N$ is the cylinder set with base $M$ and generating space $V'/ \cF^\circ$. 
 These measures $\nu_\cF$ are not independent of each others. If $\cG$ is another finite-dimensional subspace of $V'$ such that $\cF\subseteq \cG$, then for each Borel set $M$ of $V'/\cF^\circ$ we have 
\begin{align}\label{compcylicner}
\nu_\cF(M)= \nu_\cG(\pi^{-1}(M))
\end{align} 
where $\pi$ denotes the canonical map of $V'/\cG^\circ$ into $V'/\cF^\circ$ given by $\pi(L+\cG^\circ)=L+\cF^\circ,$ $L\in V'$.
Note that (\ref{compcylicner}) is a compatibility condition for the measures $\nu_\cF$. That is, if (\ref{compcylicner}) holds for a family of normalized measures $\nu_\cF$ indexed by the  finite-dimensional subspaces $\cF$ of $V$, then it comes from a cylinder measure $\mu$, defined by  (\ref{nufmun}), on $V'$ (see e.g. \cite[Chapter IV, \S 1, 5.]{gw4}. 

Cylinder measures are defined only on the algebra $\cC(V)$ of cylinder sets and they are not $\sigma$-additive in general. Hence it is natural to look for conditions on $V$ which imply a cylinder measure $\mu$ on $V'$ is $\sigma$-additive on $\cC(V)$. In this case $\mu$ extends to an "ordinary" $\sigma$-additive measure on the $\sigma$-algebra $\cB(V)$. 

%A deep result in this respect is the following {\it Minlos' theorem}. 
In what follows   nuclear operators and   nuclear spaces will play a crucial role. 
Definitions and basic facts on these notions can be found (for instance) in  \cite[Chapter I, \S 3]{gw4} or  \cite[Chapter  III, 7.]{schaefer}. We briefly repeat.

Let $(V_1,\langle \cdot,\cdot\rangle_1)$ and  $(V_2,\langle \cdot,\cdot\rangle_2)$ be Hilbert  spaces. An operator $b:V_1\to V_2$ is called {\it nuclear}  if there exist sequences $(\varphi_n)_{n\in \N}$ of $V_1$ and $\psi_n)_{n\in \N}$ of $V_2$ such that 
\begin{align*}
\sum_{n=1}^\infty \|\varphi_n\|_1\, \|\psi_n\|_2<\infty\quad \rm{and}\quad b=\sum_{n=1}^\infty \langle \cdot,\varphi_n\rangle_1 \psi_n.
\end{align*}
A positive self-adjoint operator $b:V\to V$ of a Hilbert space $(V,\langle \cdot,\cdot\rangle)$ is nuclear, equivalently of {\it trace class}, if and only if\,
$\sum_{i\in I} \langle b\eta_i,\eta_i\rangle <\infty$ for some (then for each) orthonormal basis $(\eta_i)_{i\in I}$ of $V$.

Let $(V_n,\langle \cdot,\cdot\rangle_n), n\in \N,$  be Hilbert spaces such  that $V_n\subseteq V_m$ and $\|\cdot\|_m\leq \|\cdot\|_n$ for all $n\geq m.$ The vector space $V=\cap_n V_n$, equipped with the locally convex topology defined by the family of norms $\|\cdot\|_n, n\in \N$, is called a {\it $\sigma$-Hilbert space}. Such a space $V$ is called {\it nuclear} if for each $n\in \N$ there exists an $m\in \N$ such that the embedding of the Hilbert space $(V_m,\langle \cdot,\cdot\rangle_m)$ into $(V_n,\langle \cdot,\cdot\rangle_n)$ is nuclear. 

Note that the Schwartz space $\cS(\R^d)$ is a nuclear $\sigma$-Hilbert space. Indeed, let $N$ denote the number operator with domain $\cD(T)$ in the Hilbert space $(L^2(\R^d),\langle \cdot,\cdot\rangle)$. Then $\cS(\R^d)=\cap_n V_n$, where $V_n:=\cD(T^n)$ and $\langle \cdot,\cdot\rangle_n:= \langle N^n \cdot ,N^n \cdot\rangle_n, n\in \N.$

An important and deep result  is the following {\it Minlos' theorem}. 
\begin{thm}\label{minlosthm}
Suppose that $V$ is a nuclear locally convex Hausdorff space. Then each continuous cylinder measure on $V'$ is $\sigma$-additive and extends to a measure on the $\sigma$-algebra $\cB(V)$.
\end{thm}
\begin{proof} 
If $V$ is a $\sigma$-Hilbert space, this result was proved by R.A. Minlos \cite{minlos}. A proof of this case is given in \cite[pp. 290--30]{gw4}. The general case is treated in \cite{umemura}. 
\end{proof}

Now we specialize to the case when $V$ is a Hilbert  space  with scalar product $\langle \cdot,\cdot\rangle$. Let $(b_n)_{n\in \N}$ be  a sequence of positive nuclear operators on $V$. Then the sets 
\begin{align}
U_{n,r} :=\{t \in V: \langle b_n t, t\rangle\leq r\}, \quad where~~ n\in \N, ~r >0,
\end{align} 
form a $0$-neighborhood basis of a locally convex topology on $V$, called the {\it Sasonov topology} associated with the sequence $(b_n)_{n\in \N}$. 

A cylinder measure $\mu$ on $V'$ is called  {\it continuous} with respect to the Sasonov  topology if for each $\varepsilon >0$ there exist $n\in \N$ and $\delta >0$ such that 
$\langle b_n t ,t\rangle \leq \delta$ implies  that $\mu(\{ s\in V: |\langle s,t\rangle  \geq 1\})\leq \varepsilon$. 

The following {\it Sazonov theorem}    characterizes $\sigma$-additive cylinder measures.
\begin{thm}\label{sazonov}
Suppose that $V$ is a Hilbert  space. 
A cylinder measure $\mu$ on $V'$ is $\sigma$-additive   if and only if there exists a sequence $(b_n)_{n\in \N}$ of positive nuclear operators on $V$ such that $\mu$ is continuous in the corresponding Sazanov topology.
\end{thm}
\begin{proof} \cite[Chapter IV, \S 2, Proposition 7]{gw4} or \cite{sazonov}.
\end{proof}

\section{Moment functionals with cylinder measures}\label{momentfunccylsec}
Let us introduce our setup and some notations   that will be kept in this paper. Throughout, $A$ is a unital real  commutative algebra. We suppose that  $T$ is a linear subspace of $A$ such that $T\cup\{1\}$   generates the algebra $A$ and $S$ is a vector space basis of $T$. The vector space of all linear functionals on $T$ is denoted by $T^*$. 

 By a character of $A$ we mean  an algebra homomorphism $\chi:A\to \R$ such that $\chi(1)=1.$ The set of  characters of the commutative algebra $A$ is denoted by $\hat{A}.$ We equipp $\hat{A}$ with the weak topology inherited from $A$. A linear functional $L$ on $A$ is called positive if $L(f^2)\geq 0$ for all $f\in A.$
 
Obviously, each character of $A$ is uniquely determined by its values on $T.$ Hence the map $\hat{A}\ni \to \chi\lceil T \in T^*$ is injective. Simple  examples show that this map is not surjective in general, that is,  its range $\hat{A}\lceil T\equiv\{ \chi\lceil T;\chi\in \hat{A}\,\}$ can be different from $T^*$. For notational simplicity we will not distinguish between $\chi\in \hat{A}$ and its restriction $\chi\lceil T,$ that is, with a slight abuse of notation we consider $\hat{A}\lceil T$ as a subset of $\hat{A}$.  

Let $\cT$ denote the vector space $T$ equipped  with some locally convex Hausdorff topology. As usual,  $\cT'$ is the  dual vector space of all continuous linear functionals on $\cT$.  In what follows we will deal with cylinder measures on $\cT'$. 
(For instance, if we take the finest locally convex topology on the vector space $\cT$,  then $\cT'=T^*.$)

\begin{dfn}\label{defm}
Suppose that $\cT'\subseteq \hat{A}\lceil T.$ Let $\mu$ be a cylinder measure on $\cT'$ such that the function $\chi\mapsto f(\chi)$ on $\cT'$ is $\mu$-integrable for all $f\in A$. The linear functional 
$L$ on $A$ defined by 
\begin{align}\label{defmfcy}
L(f)=\int_{{\cT'}} f(\chi) \, d\nu(\chi)\quad {\rm  for}~~ f\in A,
\end{align}
is called a \emph{moment functional} on $A$ and $\mu$  a \emph{representing cylinder measure} of $L.$
\end{dfn}
  Since $\cT'\subseteq\hat{A}\lceil T$, equation (\ref{defmfcy}) means that $L$ is an integral over characters of $A$ with respect to the cylinder measure $\mu$. 
  
 The following example shows how the classical moment problem fits into the above setup.

\begin{exa}
Let $A=\R[x_1,\dots,x_d]$ and $T:=\{a_1x_1+\dots+a_dx_d: a_1,\dots,a_d\in \R\}.$  Since $\hat{A}$ consists of  point evaluations at points of $\R^d$, we have $\hat{A}\lceil T=T^*=\cT'\cong \R^d$ for any locally convex Hausdorff topology on $T.$ In this case, the moment functionals according to  Definition \ref{defm} are just the  "usual" moment functionals on $\R[x_1,\dots,x_d].$ 
\end{exa} 
  All integral representations of functionals $L$ developed below remain valid without the assumption $\cT'\subseteq \hat{A}\lceil T$. The latter is only used to  conclude that $L$ is an integral over characters of $A$, so  that $L$  is  a moment functional. 
  
  In the rest of this section  we  discuss the condition $\cT'\subseteq \hat{A}\lceil T$ in Definition \ref{defm}.

If the character set $\hat{A}$ separates the points of $T$ (that is,  $t\in T$ and  $\chi(t)=0$ for all $\chi\in \hat{A}$ imply  $t=0$), this assumption is fulfilled if we equip $T$ with  the weak topology $\sigma(T,\hat{A}\lceil T)$ of the dual pairing $(T,\hat{A}\lceil T)$. This topology is Hausdorff by the separation assumption and we have $\cT'=\hat{A}\lceil T$ by a standard result on locally convex spaces.

Now let us consider the case when $\hat{A}$ does not separate the points of $A.$ Clearly,  
\begin{align}\label{radch}
{\rm Rad}_{\rm ch}(A):=\{ f\in A: \chi(f)=0 ~~ {\rm for~ all}~~\chi\in \hat{A}\, \}
\end{align}
is an ideal of $A$ and the quotient algebra $A/ {\rm Rad}_{\rm ch}(A)$ 
has the property that the characters separate the points. 

If $L$ is a positive functional on $A$ and $L$ vanishes on   $f\in {\rm Rad}_{\rm ch}(A),$  then $L$ induces a positive functional on  the quotient algebra $A/ {\rm Rad}_{\rm ch}(A)$ and we can study the moment problem there. Since  this quotient algebra characters separate points, we can proceed as in the preceding remark.

The following   question came up some years ago  in a discussion with Y. Savchuk:\smallskip

{\it  (*)~ Do all positive functionals on $A$ vanish on  ${\rm Rad}_{\rm ch}(A)$?}  
\smallskip

For finitely generated algebras this follows easily from Haviland's theorem.
As noted by T. Bisgaard \cite{bisgaard}, it is true if $A$ is the semigroup algebra $\R[S]$ of a commutative unital semigroup $S$. 
For general algebras the answer is negative as the follwoing example shows.

\begin{exa}\label{arens}
Let $A$ be the  Arens algebra  $L_\R^\omega(0,1):=\cap_{p\geq 1} L_\R^p(0,1),$ where $L_\R^p(0,1)$ is the real $L^p$-space with respect to the  Lebesgue measure. Then $A$ is a commutative unital real  algebra which has no character (see \cite{Arens} or \cite[Example 2.5.10]{Sch1990}), so that ${\rm Rad}_{\rm ch}(A)=A$. Let  $\varphi\in C_0(0,1), \varphi\neq 0.$ Then 
\begin{align*}
L_\varphi(f)=\int_0^1 f(x)|\varphi(x)|^2 dx, \quad f\in A,
\end{align*}
is a nonzero positive functional $L_\varphi$ on $A$. Obviously, it does not vanish on ${\rm Rad}_{\rm ch}(A).$
\end{exa}
\begin{rem}
Let $C$ be  a (real or complex)   unital $*$-algebra. The  $*$-radical ${\rm Rad}(C)$ is usually defined as  the ideal  of  $c \in C$ such that $L(c^*c)=0$ for all positive linear functionals $L$ on $C$, or equivalently, of elements of $C$ which are annihilated by all (possibly unbounded) $*$-representations of $C$. 

The commutative real algebra $A$ is a $*$-algebra with identity map as involution. For such an algebra $A$ the ideal
${\rm Rad}_{\rm ch}(A)$ defined by (\ref{radch}) can be considered as  another version of the radical. Then $(*)$ is the question when  ${\rm Rad}_{\rm ch}(A)$ coincides with ${\rm Rad}(A)$.For the algebra $A$ in Example \ref{arens} we have ${\rm Rad}_{\rm ch}(A)=A$ and\, ${\rm Rad}(A)=\{0\}$.
\end{rem}
\section{Integral Representation of Linear Functionals}\label{haviland}
 
 Now we begin with some preparations for our main result (Theorem \ref{mpscylinder}).
 
  Let $Q$ denote the unital monoid generated by the elements $1+t^2$, where $t\in S$, and let $B:=Q^{-1}A$ be 
the corresponding algebra of fractions. 
\begin{lem}\label{extensionchar}
Each character $\chi\in \hat{A}$ extends uniquely to a character, denoted again by $\chi$ with a slight abuse of notation, of $B=Q^{-1}A$ by setting $\chi(q^{-1})=\chi(q)^{-1}, q\in Q$.
The map $\hat{B}\ni \chi\mapsto \chi\lceil T\in T^*$ is injective. 
\end{lem}
\begin{proof}
Let $\chi\in \hat{A}$. Let $q\in Q, q\neq 1.$ Then $q$ is of the form $q=(1+t_1)\dots(1+t_k^2)$ with $t_1,\dots,t_k\in T$ and 
\begin{align}\label{cht}
\chi(q)=(1+\chi(t_1)^2)\cdots(1+\chi(t_k)^2)>0.
\end{align}
Any element $b$ of $Q^{-1}A$ is of the form $b=q^{-1}a$ with $q_1\in Q$ and $a\in A$. Note that $\chi(q)>0$ by (\ref{cht}).
 We define $ \chi(b)=\chi(q)^{-1}\chi(a)$. One easily checks that this definition is independent of the particular representation $b=q^{-1}a$ of the element $b\in B$ and that $\chi$ becomes a character on the algebra $B=Q^{-1}A$. This completes the proof of the first assertion.
 
  Since $T$ generates the algebra $A$, each character of $A$, by the preceding also  each character of $B$, is uniquely determined by its values on $T$. This implies the injectivity of the map  $ \chi\mapsto \chi\lceil T$. 
 \end{proof}
Thus, the character of $A$ are precisely the restriction of characters of $B$. As usual we will consider elements of $B$ as functions on the character set $\hat{B}$ and write $b(\chi):=\chi(b)$ for $\chi\in \hat{B}$ and $b\in B$.

The crucial technical result for our approach  is the following lemma.
Similar   algebras of fractions have been used in \cite{Sch2010} and also in \cite{gik}.
\begin{lem}\label{finitecase}
Let $\cF$ be a finite-dimensional linear subspace of $T$ and let $\{t_1,\dots,t_k\}$ be a vector space basis   of $\cF$. Let  $B_F$ denote the unital subalgebra of $B$ generated by\, $t_i$ and $a_i:=(1+t_i^2)^{-1}$ for $i=1,\dots,k$. For  each positive linear functional $L$ on $B_F$ there exists a unique Radon measure $\nu_F$ on $\cF^*$ 
such that
\begin{align}\label{repposbf}
L(b)=\int_{{\cF^*}}b(\chi) \, d\nu_\cF(\chi)\quad {\rm  for}~~ b\in B_F.
\end{align}
\end{lem}
Before we prove this lemma let us rewrite the integral (\ref{repposbf}) in terms of coordinates on the vector space $\cF^*$. This  builds the bridge to the  moment problem on $\R^k$. Define functionals $\chi_i\in \cF^*$ by $\chi_i(t_j)=\delta_{ij}$, where $i,j=1,\dots,k$. Clearly, $\{\chi_1,\dots,\chi_k\}$ is a vector space basis of  $\cF^*$. It is  the dual basis to the basis $\{t_1,\dots,t_k\}$ of $\cF$. Hence each $\chi\in \cF^*$ can be written as 
\begin{align*}
\chi=x_1\chi_1+\dots+x_k\chi_k \quad {\rm with}~~~ x(\chi):=(x_1,\dots,x_k)\in \R^k
\end{align*} uniquely determined by $\chi$. We shall write   $\tilde{b}(x(\chi)):=b(\chi)$ for $\chi\in \cF^*$ and $b\in B_F$. Then $\tilde{b}$ is a function on $\R^k$. Let $\tilde{\nu}_\cF$ denote the  Radon measure  on $\R^k$ defined by $\tilde{\nu}_\cF(x(\chi))):=\nu_\cF(\chi)$. Then we have 
\begin{align}\label{coordintegral}
\int_{\cF^*}\, b(\chi)\, d\nu_\cF(\chi)=\int_{\R^k}\, \tilde{b}(x)\, d\tilde{\nu}_\cF(x) \quad {\rm for}~~b\in B_F.
\end{align}

{\it Proof of Lemma \ref{finitecase}:}

The existence assertion will be derived from the general fibre theorem proved in \cite{schm16}. 
For $i=1,\dots,k$, we set $b_i:=t_i(1+t_i^2)^{-1}$. Clearly, $b_i\in B_F$. Let $\chi\in \widehat{B_F}$. From the identity  $(a_i-1/2)^2+b_i^2=1/4$  we conclude that 
\begin{align*}
(\chi(a_i)-1/2)^2+\chi(b_i)^2=1/4.
\end{align*}
Therefore, $|\chi(a_i-1/2)|\leq 1/2$, hence $|\chi(a_i)|\leq 1$, and $|\chi(b_i)|\leq 1/2$. This shows that the $2k$ functions $h_i=a_i$ and $h_{k+i}:=b_i$, where $i=1,\dots,k$,  on $\widehat{B_F}$ are bounded.   We apply the fibre theorem to the bounded functions $h_1,\dots,h_{2k}$ and the preordering $\cP:=\sum (B_F)^2$.  Note that $\cK(\cP)=\widehat{B_F}$. 

Let us fix numbers $\lambda_i \in h_i(\cK(\cP)), i=1,\dots,2k$, and set $\lambda=(\lambda_1,\dots,\lambda_{2k})\in \R^{2k}$. For $\chi\in \widehat{B_F}$ and $i=1,\dots,k$, we  obtain 
$$(1+\chi(t_i)^2)\chi(h_i)=\chi(1+t_i^2)\chi(a_i)=\chi(1+t_i^2)a_i)=\chi(1)=1,$$ so that $\chi(h_i)>0$ and therefore $\lambda_i>0$. Let $\cI_\lambda$ denote the ideal of $B_F$ generated by the functions $h_i-\lambda_i$, $i=1,\dots,2k.$ In the quotient algebra $B_F/ \cI_\lambda$ we have  $h_i=(1+t_i^2)^{-1}=\lambda_i$ and $h_{k+i}=t_i(1+t_i^2)^{-1}=\lambda_{k+i},$
 so that $t_i=\lambda_{k+i}\lambda_i^{-1}$ for $i=1,\dots,k$. Hence the algebra   $B_F/ \cI_\lambda$ is $\R$, so it  obviously obeys property (MP). Thus, all assumptions of the fibre theorem  are satisfied.  By the fibre theorem \cite{schm16} (or equivalently, by the implication (iii)$\to$(i) of  \cite[Theorem 13.10]{schm2017}), the preordering $\sum (B_F)^2$ of the algebra $B_F$ has property (MP). This means that the positive linear functional $L$ on $B_F$ is an integral of some Radon measure $\nu$ on $\widehat{B_F}$. From Lemma \ref{extensionchar}, applied to $B_F$ instead of $B$, it follows that $\widehat{B_F}\ni \chi \to \chi\lceil \cF\in \cF^*$ is injective. If $\nu_\cF$ denotes the pushforward of the measure $\nu$ under this mapping, we obtain the desired integral representation (\ref{repposbf}).

Finally, we prove the uniqueness assertion. 
 Let $\mu_F$ and $\nu_F$ be two Radon measures on $\cF^*$ for which (\ref{repposbf}) is satisfied.  Set
 $h_i:=(1+t_i^2)^{-1}$ and $h_{k+i}:=t_i(1+t_i^2)^{-1},$   $i=1,\dots,k$. Let $C_F$ denote the subalgebra of $B_F$ generated all elements $h_{i_1}\cdots h_{i_k}$, where $i_j\in \{j,k+j\},$ and let $D_F$ be the algebra of functions $\tilde{h}$, where $h\in C_F$. Clearly,
 $\widetilde{h_i}(x)=(1+x_i^2)^{-1}$ and $\widetilde{h_{k+i}}(x)=x_i(1+x_i^2)^{-1}$ for $x=(x_1,\dots,x_k)\in \R^k$ and $i=1,\dots,k$. Further, if $h:=h_{i_1}\cdots h_{i_k}$, then $\tilde{h}=\widetilde{h_{i_1}}\dots \widetilde{h_{i_k}}$. This implies that all functions of $D_F$  vanish at infinity. Clearly, the algebra $D_F$ separates the points of $\R^k$.  For $g:=h_1\cdots h_k$ we obviously have $\tilde{g}\neq 0$ on $\R^k$. Clearly,   $D_F$ is an algebra of continuous functions on the locally compact space $\R^k$. As shown by the preceding, this algebra satisfies the assumptions of the Stone-Weierstrass theorem  (see e.g. \cite[Chapter V, Corollary 8.6]{conway}). By this theorem, for each function $f\in C_0(\R^k)$ there exists a sequence $(f_n)_{n\in \N}$ of elements from $C_F$ such that the sequence $(\widetilde{f_n})_{n\in \N}$ of $D_F$ converges to $f$ uniformly on $\R^k$. By (\ref{repposbf}) and  (\ref{coordintegral}), we have
 \begin{align}\label{mutildemuf}
\int_{\R^k}\, \widetilde{f_n}(x)\, d\tilde{\mu}_F(x)= \int_{\cF^*}\, f_n(\chi)\, d\nu_\cF(\chi)= \int_{\cF^*}\, f_n(\chi)\,d\nu_F=\int_{R^k}\, \widetilde{f_n}(x)\, d\tilde{\nu}_F(x).
 \end{align}
 Passing to the limit in (\ref{mutildemuf}) we get $\int f(x)\, d\tilde{\mu}_F=\int f(x)\, d\tilde{\nu}_F$. Since this holds for all $f\in C_0(\R^k)$,  we conclude  that $\tilde{\mu}_F=\tilde{\nu}_F$. Hence $\mu_F=\nu_F.$
$ \hfill \Box$
\medskip

\begin{thm}\label{mpalgebraB}
Let $L$ be a positive linear functional on $B$ such that $L(1)=1$. 
Then there exists a unique cylinder measure $\mu$ on $\cT'$ such that 
\begin{align}\label{replcylincer}
\tilde{L}(b)=\int_{{\cT'}} b(\chi) \, d\mu(\chi)\quad {\rm  for}~~ b\in B.
\end{align}
\end{thm}
\begin{proof}
Let $\cF$ be a finite-dimensional linear subspace of $\cT$. By Lemma  \ref{finitecase} there exists a unique Radon measure $\nu_\cF$ on $\cF^*\cong \cT'/\cF^\circ$ such that (\ref{repposbf}) holds. Since $1\in B_F$ and $L(1)=1$ by assumption, each measure $\nu_\cF$ is normalized. If $N$ is a  cylinder set with base $M$ and generating space $\cT'/ \cF^\circ$, we define $\mu(N):= \nu_\cF(M)$.  As noted above, to show that $\mu$ is a well-defined cylinder measure on $\cT'$ it suffices to verify the compatibility condition  (\ref{compcylicner}).

Let  $\cG$ be another finite-dimensional linear subspace of $\cT$ such that $\cF\subset \cG$. We  choose a basis $G:=\{f_1,\dots,f_k,f_{k+1},\dots,f_n\}$, $k<n$, of the vector space $\cG$ such that $F:=\{f_1,\dots,f_k\}$ is a basis of $\cF$. 
Let $\nu_\cG$ be the corresponding measure for $\cG$ according to Lemma \ref{finitecase}.  Further, we denote by  $\tau $  the Radon measure on $\R^k$ given by $\sigma(M)=\tilde{\nu}_\cG(M\times \R^{n-k})$, where $M$ is a Borel set of $\R^k.$
We use the notation and some arguments from the proof of Lemma \ref{finitecase}. Let $b\in B_F$. Since $F\subseteq G$, we  have $b\in B_G$. Using (\ref{repposbf}) and  (\ref{coordintegral}) for both $\nu_\cF$ and  $\nu_\cG$ we derive 
\begin{align}\label{equ1}
\int_{\R^k}\, \tilde{b}(x)\, d\tilde{\nu}_F(x) &= \int_{\cF^*}\, b(\chi)\, d\nu
(\chi)=L(b)= \int_{\cG'}\, b(\chi)\,d\nu_F(\chi)\\&=\int_{\R^n}\, \tilde{b}(y)\, d\tilde{\nu}_G(y)=\int_{R^k}\, \tilde{b}(x)\, d\sigma(x).\label{equ2}
 \end{align}
Here the last equality follows from the fact that for $b\in B_F$ the function $\tilde{b}$ on $\R^n$ depends only on the first $k$ coordinates. By Lemma \ref{finitecase},  $\mu_F$ is th unique Radon  measure on $\cF^*$ satisfying(\ref{repposbf}). Hence $\tilde{\nu_\cF}$ is the unique Radon measure on $\R^k$ such that 
$\int \tilde{b}\, d\tilde{\nu}_\cF =L(b)$ for all $b\in B_F$. Therefore, it follows from  (\ref{equ1})--(\ref{equ2}) that $\tilde{\nu}_\cF=\sigma.$ Hence $\tilde{\nu}_\cF(M)=\sigma(M)=\tilde{\nu}_\cG(M\times \R^{n-k})$ for all Borel sets $M$ of $\R^k$. It is easily verified that the latter is just condition  reformulated in terms of coordinates. This completes the proof of the fact that $\mu$ is a well-defined cylinder measure. 

Let $b\in B.$ Then $b$ is contained in some algebra $B_F$, so that (\ref{repposbf}) holds.  Since $\mu(N):= \nu_\cF(M)$, this implies that (\ref{replcylincer}) holds.
The uniqueness of $\mu$ follows at once from the uniqueness of the measures $\nu_\cF$ stated in Lemma  \ref{finitecase}.
\end{proof}
Now we turn to the moment problem on the algebra $A$. Define
 \begin{align*}
 A_+=\{ f\in A: f(\chi)\geq 0\quad {\rm for}~~~\chi\in \hat{A}\,\}.
 \end{align*}
 Clearly, each moment functional $L$ is $A_+$-positive (by (\ref{defmfcy})) and satisfies $L(1)=1$ (since $\mu(\cT')=1)$ by Definition \ref{cmd}(ii)).

 Further, if $L$ is a positive functional on $A$, we will say that {\it Carleman's condition} holds for an element $t\in A$ if 
 \begin{align}\label{carleman}
 \sum_{n=1}^\infty L(t^{2n})^{-\frac{1}{2n}} =+\infty.
 \end{align}
 (Note that $L(t^{2n})\geq 0$, because the functional $L$ is positive.)
 
 The mai result of this paper is the following theorem.
 \begin{thm}\label{mpscylinder}
Let $L$ be a linear functional on $A$ such that $L(1)=1$. Suppose that  one of the following assumptions $(\rm i)$ or $(\rm ii)$ is satisfied:
 \begin{itemize}
\item[(i)]\, $L$ is $A_+$-positive, that is, $L(a)\geq 0$ for  $a\in A_+$.
\item[(ii)]\, $L$ is positive, that is, $L(a^2)\geq 0$ for $a\in A$, and     Carleman's condition (\ref{carleman}) holds for all $t$ of the vector space basis $S$ of $T$.
 \end{itemize}Then there exists a cylinder measure $\mu$ on $\cT'$ such that
 \begin{align}\label{mpalgebraA}
L(f)=\int_{{\cT'}} f(\chi) \, d\mu(\chi)\quad {\rm  for}~~ f\in A.
\end{align}
If  there exists a continuous seminorm $q$ on $\cT$ such that 
\begin{align}\label{contL}
L(t^2)\leq q(t)^2\quad{\rm for}\quad t\in \cT,
\end{align} then there is a continuous cylinder   measure
$\mu$  on $\cT'$ such that (\ref{mpalgebraA}) holds.

Further, if in addition $\cT'\subseteq \hat{A}\lceil T$, then $L$ is a moment functional on $A$ according to Definition \ref{defm}.
\end{thm}
 \begin{proof}
First  assume (i). Let $B=Q^{-1}A$ the algebra of fractions defined above. Since the characters of $A$ are restrictions to $A$ of characters of $B$  by Lemma \ref{extensionchar}, we have $B_+\cap A=A_+$. Since all elements $a_i,b_i$ are bounded on $\hat{B}$ as shown in the proof of Lemma \ref{finitecase}, so are all elements of $Q^{-1}$. Hence for each $b\in B$ there exists $a\in A$ such that $b(\chi)\leq a(\chi)$ for all $\chi\in \hat{B}$. Therefore, the $A_+$-positive functional $L$ can be extended to a $B_+$-positive linear functional $\tilde{L}$ on $B$. From Theorem \ref{mpalgebraB}, applied to $\tilde{L}$, it follows that there exists a cylinder measure $\mu$ on $\cT'$ such that (\ref{mpalgebraA}) is satisfied.

 Now we suppose that  (ii) is satisfied. Let $\cF$ be a finite-dimensional linear subspace of $T$ and  $F:=\{f_1,\dots,f_k\}$  a basis of $\cF$. Let $A_F$ denote the unital subalgebra of $A$ generated by $f_1,\dots,f_k$. We define a linear functional $L_F$ on $\R[x_1,\dots,x_k]$ by $$L_F(p(x_1,\dots,x_k))=L(p(f_1,\dots,f_k)).$$ Since $L_F(x_j^{2n})=L(f_j^{2n})$ for $j{=}1,\dots,k$ and $n\in \N$, it follows from (ii) that $L_F$ is a positive functional on $\R[x_1,\dots,x_k]$ satisfying the multi-variate Carleman condition. Therefore, by Nussbaum's theorem, $L_F$ is a determinate moment functional on $\R[x_1,\dots,x_k]$, that is, $L_F$ has a unique representing Radon measure on $\R^k$. Therefore, there exists a unique Radon measure $\nu_\cF$ on  $\cF^*\cong \R^k$ such that 
 $$L(a)=\int_{\cF^*} a(\chi)\, d\nu_\cF(\chi)\quad{\rm for}~~a\in A_F.
 $$
  Proceeding  as in the proof of Theorem \ref{mpalgebraB} it follows that the family of measures $\nu_F$ give a well-defined cylinder measure on $\cT'$. The uniqueness of  measures $\nu_F$ used in a crucial manner in the proof of Theorem \ref{mpalgebraB}  follows now from the uniqueness of   representing measures of functionals $L_F$ because of the Carleman condition.

Now suppose that the continuity assumption (\ref{contL}) is fulfilled. Let $\varepsilon>0$ and $a>0$ be given. Set $\delta:=a\sqrt{\varepsilon}$ and   $U:=\{ t\in T: q(t)\leq \delta \}.$ Then, for $t\in U$, 
\begin{align*}
\mu( \{ \chi \in \cT': & |\chi(t)|\geq 1\}) \leq a^{-2}\int_{\cT'} \chi(t)^2 d\mu(\chi) \\& =a^{-2} \int_{\cT'} \chi(t^2) d\mu(\chi)= a^{-2}L(t^2) \leq a^{-2}q(t)^2\leq a^{-2}\delta^2=\varepsilon.
\end{align*}  
Thus, the condition in Definition \ref{defconti} is satisfied, so $\mu$ is continuous.

If  $\cT'\subseteq \hat{A}\lceil T$,  then Definition \ref{defm} is satisfied, so $L$ is a moment functional on $A.$
\end{proof}

\section{Nuclearity assumptions and representing measures}\label{nucasssec}

If we assume  the nuclearity of the space $\cT$ in Theorem \ref{mpscylinder} , then we get an "ordinary" $\sigma$-additive representing measure of $L$ rather than a cylinder measure. 
\begin{thm}\label{nucearass}
 Let $L$ be a linear functional on $A$ such that   one of the assumptions $(\rm i)$ or $(\rm ii)$ in  Theorem \ref{mpscylinder} holds. Suppose that $\cT$ is a nuclear locally convex Hausdorff space and there exists a continuous seminorm $q$ on $\cT$ such that 
\begin{align}\label{contL}
L(t^2)\leq q(t)^2\quad{\rm for}\quad t\in \cT.
\end{align} 
Then there exists a  measure  $\mu$ defined on the $\sigma$-algebra  $\cB(\cT)$ such that
 \begin{align}\label{mpalgebraA}
L(f)=\int_{{\cT'}} f(\chi) \, d\mu(\chi)\quad {\rm  for}~~ f\in A.
\end{align}
If $\cT'\subseteq \hat{A}\lceil T$, then $L$ is a moment functional with representing measure $\mu$.
\end{thm}
\begin{proof}
Assumption (i) implies that $L$ is positive. Therefore, if $L(1)=0$, then $L=0$ by the Cauchy-Schwarz inequality, so the assertion holds trivially with $\mu=0$.

 Suppose that $L(1)\neq 0.$  Then $L(1)>0$ by (i) or (ii) and Theorem \ref{mpscylinder}  applies to the functional $L(1)^{-1}L$ on $A$. Since the space $\cT$ is nuclear, the corresponding cylinder measure is a measure by Minlos' Theorem
\ref{minlosthm}.
\end{proof}

By Minlos theorem,  the nuclearity of the space $\cT$ implies  that  each cylinder measure  on $\cT'$ is indeed a measure. A necessary and sufficient  condition when a given cylinder measure is a measure   is provided by Sazonov's theorem \ref{sazonov}. We do not restate the corresponding result in the present setting.

As an important special case  we can take the whole algebra $A$ as vector space $T.$ Then,  $\hat{A}=\hat{A}\lceil T$, so the assumption   $\cT'\subseteq \hat{A}\lceil T$ is trivially satisfied and Theorems \ref{mpscylinder}  and \ref{nucearass} provide integral representations of the corresponding functional $L$ over characters. 

Now we begin with the preparations for the next main result (Theorem \ref{intcharalgebra}).

A real algebra $C$ equipped with a locally convex topology is called a {\it topological algebra} if for each $b\in C$  the  mappings $c\to bc$ and $c\to ab$ are continuous on $C$.

A locally convex space $V$ is called {\it barrelled} if each absolutely convex closed absorbing subset of $V$ is a zero neighbourhood \cite[Chapter II, Section 7]{schaefer}. For instance, each complete metrizable locally convex space is barrelled.

The following  lemma is a known result  (see e.g. \cite[Corollary 3.6.]{Sch1990}) adapted to the present context.
\begin{lem}\label{barrecont}
Suppose that $C$ is a barrelled topological real algebra. If $L$ is a continuous positive linear functional on $C$,  there exists a continuous seminorm $q$ on $C$ such that $L(c^2)\leq q(c)^2$ for all $c\in C$.
\end{lem}
\begin{proof}
First we prove the equality 
\begin{align}\label{cauchyscwar}
L(c^2)^{1/2}=\sup\, \{ |L(cb)|: b\in C; L(b^2)=1\}.
\end{align} Let $M_c$ denote the supremum on the right hand side of (\ref{cauchyscwar}). Since $L$ is positive, the Cauchy-Schwarz inequality yields\,  $|L(cb)|\leq L(c^2)^{1/2}L(b^2)^{1/2}=L(c^2)^{1/2}.$ Hence   $M_c\leq L(c^2)^{1/2}$. Further, if $L(c^2)=0$, then $M_c=0$ by this equality. Now suppose that $L(c^2)>0$. Then  $b:=cL(c^2)^{-1/2}$ satisfies  $L(b^2)=1$ and $|L(cb)|=L(c^2)^{1/2}$, so that $L(c^2)^{1/2}\leq M_c$. Putting the preceding together we have proved (\ref{cauchyscwar}).

Now we set $U:=\{ c\in C: L(c^2)^{1/2}\leq 1\}.$ Obviously, $U$ is absorbing. From (\ref{cauchyscwar}) we easily conclude that the set $U$ is absolutely convex. 

We prove that $U$ is closed in $C$. Let $(c_i)_{i\in I}$ be a net from $U$ converging to some element $c\in C$. Fix $b\in C$ such that $L(b^2)=1$. Since $L$ is continuous and the map $a\to ab$ on $C$ is also continuous (because $A$ is a topological algebra), we conclude that 
$\lim_i L(c_ib)=L(cb)$. From  $c_i\in U$ and (\ref{cauchyscwar}) it follows  that  $|L(c_ib)|\leq 1$, so that $|L(cb)|\leq 1.$ Therefore, again by (\ref{cauchyscwar}), $L(c^2)^{1/2}\leq1$, that is, $c\in U$. This shows that $U$ is closed. 

Since the locally convex space $C$ is barrelled, $U$ is a zero neighbourhood in $C$. Hence there exists a continuous seminorm $q$ on $C$ such that $\{ c\in C: q(c)\leq 1\}\subseteq U.$ This  implies that $L(c^2)^{1/2}\leq q(c)$ for all $c\in C$. (In fact, $r(c):=L(c^2)^{1/2}, c\in C,$ defines  a seminorm $r$ on $C.$ Hence  $r$ is continuous on $C$ and we could  take $r=q$.)
\end{proof}
Under assumption (i) the next result was proved  by Borchers and Yngavson \cite{by}. 
\begin{thm}\label{intcharalgebra}
Suppose that the commutative real unital algebra $A$ carries a nuclear barrelled locally convex covex topology such that it is a topological algebra. Let $L$ be a continuous linear functional on $A$ satisfying  one of the following assumptions:
 \begin{itemize}
\item[(i)]\, $L$ is $A_+$-positive.
\item[(ii)]\, $L$ is positive and   Carleman's condition (\ref{carleman}) holds for all $t$ of some vector space basis of $A$.
 \end{itemize}
 Then there exists a measure $\mu$ on the $\sigma$-algebra $\cB(\hat{A})$ such that
\begin{align}
L(a)=\int_{\hat{A}} a(\chi)\, d\mu(\chi)\quad {\rm for}~~~a\in A.  
\end{align}
\end{thm}
\begin{proof}
Since $A$ is a barrelled topological algebra,  by Lemma \ref{barrecont} there is a continuous seminorm $q$ on $A$ such that $L(a^2)\leq q(a)^2$ for $a\in A$.
Therefore, setting $T=A$,  all assumptions of Theorem \ref{nucearass} are fulfilled and the assertion follows from Theorem \ref{nucearass}.
\end{proof}

\section{Moment problem for  symmetric algebras over vector spaces}\label{mpsymmetricalgebras}

Let $V$ be a real vector space. First we recall the definition of the symmetric algebra $S(V)$ over $V$. 

For $n\in \N$,  $V_n=V\otimes \cdots\otimes V$ denotes the $n$-fold tensor product of $V$ and  $S_n(V)$ is the subspace of symmetric tensors. For $v=\sum_i v_{1i}\otimes\cdots\otimes v_{ni}\in V_n$ we set $$s(v):=\frac{1}{n!} \sum_\theta\sum_i v_{\theta(1)i}\otimes \cdots\otimes v_{\theta(n)i},$$
where the summation is over all permutation $\theta$ of $\{1,\dots,n\}.$ Set $S_1(V):=V$ and $S_0(V):=\R$. 
Let $S(V)$ be the direct sum of vector spaces $S_n(V)$,  $n\in \N_0$. Then $S(V)$ becomes a  commutative unital real algebra with product determined by
$$(v\cdot w)_m=\sum_{k+n=m} s(v_k\otimes w_n) \quad {\rm for}~~ v=(v_k),w=(w_k)\in S(V)$$ with the obvious interpretations $v_0\otimes w_n=v_0w_n$, $v_n\otimes w_0=w_nv_0$,  $s(v_0\otimes w_0)=v_0w_0$ and the unit element $1=(1,0\dots).$ That is, $S(V)$ is the  free commutative real unital algebra over the vector space $V$; it is called the {\it symmetric algebra} over $V$. 

Let $V^*$ be the dual vector space of $V$. Since $S(V)$ is the free commutative algebra over $V$, each functional $L\in V^*$ extends uniquely to a character of $S(V)$. Hence the restriction $L\to L\lceil V$ is a bijection of $\widehat{S(V)}$ and $V^*$. That is, the linear functionals on $V$ a are in one-to-one  correspondence with characters of the algebra $S(V).$ 
Letting $A=S(V)$ and $T=V$, we are in the setup described at the beginning of   Section \ref{momentfunccylsec} and we have $\hat{A}\lceil T=V^*.$ 

The second assertion of the next theorem is the well-known solution of the moment problem over nuclear spaces. It was first obtained independently in \cite{by} and \cite{cs} by using different methods. An elegant approach based on Choquet theory was sketched in \cite{heg} and ealaborated in \cite[Section 15.1]{Sch1990}.
%\end{document}
\begin{thm}\label{mpsymmetrc}
Let $V$ be a real vector space and let $L$ be a linear functional on $S(V)$ such that $L(1)=1$. Suppose one of the assumptions $(\rm i)$ or $(\rm ii)$ in  Theorem \ref{mpscylinder} is satisfied. Then $L$ is a moment functional (as in Definition \ref{defm} with $\cT'=V^*$), so there exists a cylinder measure on $V^*$ such that
 \begin{align}\label{mpalgebraA}
L(f)=\int_{{V^*}} f(\chi) \, d\mu(\chi)\quad {\rm  for}~~ f\in S(V).
\end{align}
Suppose in addition that $V$ is equipped with a nuclear locally convex Hausdorff topology and there exists a continuous seminorm $q$ on $V$ such that 
\begin{align}\label{contL}
L(t^2)\leq p(t)^2\quad{\rm for}\quad t\in V.
\end{align}
 Then there exists a    measure
$\mu$  on the $\sigma$-algebra $\cB(V')$ of\, $V'$ such that 
\begin{align}\label{mpalgebraA}
L(f)=\int_{{V'}} f(\chi) \, d\mu(\chi)\quad {\rm  for}~~ f\in S(V).
\end{align} 
\end{thm}
\begin{proof}
For the first assertion  we equipp the vector space $V$ with its finest locally convex topology. Then $\cT'=V^*=\hat{A}\lceil T$, so $L$ is a moment functional by the first assertion of  Theorem \ref{mpscylinder}.

If $V$ is a nuclear space and (\ref{contL}) holds, then $\cT'=V'\subseteq \hat{A}\lceil T$ and the second assertion of   Theorem \ref{mpscylinder} applies.
\end{proof}
\begin{exa}
Let $V$ be a vector space with a countable Hamel basis. Then $S(V)$ is (isomorphic to) the polynomial algebra $\R[x_1,\dots,x_n,\dots]$ in countably many indeterminants $x_n, n\in \N.$ Further, the finest locally convex topology on $V$ is nuclear and each seminorm (in particular, the seminorm  $t\to L(t^2)^{1/2}$) is continuous in this topology. Thus, Theorem \ref{mpsymmetrc} applies, hence each $S(V)_+$-positive linear functional on $S(V)$ is a moment functional with representing measure on $\cB(V')$. In a slightly different  formulation this result was proved in \cite{gik}.
\end{exa}
\section{Gaussian measures}\label{Gausssec}
Let us begin with the finite-dimensional case. Let $(\cdot,\cdot)$ be a scalar product on $\R^d$ and let $dx$ be the Lebesgue measure on $\R^d$ for this scalar product. For a Borel subset $M$ of $\R^d$ we define
\begin{align}
\nu(M)=(2\pi)^{-d/2}\int_M e^{-(x,x)/2} dx.
\end{align}
 Then $\nu$ is a Radon measure on $\R^d$, called the (standard) {\it Gaussian measure associated with}  $(\cdot,\cdot).$ It can be nicely characterized in terms of its Fourier transform: $\nu$ is the unique probability Radon measure on $\R^d$ satisfying
\begin{align*}
\int_{\R^d} e^{{\ii} (x,y)} d\nu(x) = e^{-(y,y)/2}~~~{\rm for}~~y\in \R^d.
\end{align*}
Obviously,  $\int_{\R^d} p(x) d\nu(x)$ is finite for each polynomial $p\in \R[x_1,\dots,x_d]$.
\smallskip

Now we suppose that $(\cdot,\cdot)$ is a scalar product on a locally convex space $V$.

Let $\cF$ be a finite-dimensional subspace of $V$. We identify $\cF$ with some $\R^d$ and denote by $\nu_\cF$ the Gaussian measure on $\cF$ associated with the restriction of the scalar product $(\cdot,\cdot)$ to $\cF$. The finite-dimensional Hilbert space $(\cF,(\cdot,\cdot))$ is canonically isomorphic to its dual and so to $V'/ \cF^\circ$. Let $I_\cF$ denote this isomorphism. Then $ \mu_\cF(\cdot):= \nu_\cF(I_\cF^{-1}(\cdot))$ defines a Radon measure on   $V'/ \cF^\circ$. These measures $\mu_\cF$ satisfy the compatibility condition (\ref{compcylicner}), so they define a cylinder measure $ \mu:=\mu_{(\cdot,\cdot)}$ on $V'$,  called the  {\it Gaussian cylinder measure} associated with the scalar product $(\cdot,\cdot).$ 

If the scalar product $(\cdot,\cdot)$ is continuous on $V\times V$,  it follows easily that $\mu_{(\cdot,\cdot)}$ satisfies the  continuity condition in Definition (\ref{defconti}), so $\mu_{(\cdot,\cdot)}$ is continuous. Details and proofs for the preceding construction can be found e.g. in \cite[Ch. IV, \S 3]{gw4}.
\begin{thm}\label{gaussmf}
Let $V$ be a real locally convex Hausdorff space. 
Suppose that 
$(\cdot,\cdot)$ is a scalar product on  $V$ which is continuous on $V\times V$. 
Then  $\mu_{(\cdot,\cdot)}$ is a continuous cylinder measure on $V'$ and there exist a determinate moment functional $L$ on $A=S(V)$ such that 
\begin{align}\label{defLmuscal}
L(f)=\int_{{\cT'}} f(\chi) \, d\mu_{(\cdot,\cdot)}(\chi)\quad {\rm  for}~~ f\in S(V).
\end{align}
If the locally convex space $E$ is nuclear, then  $\mu_{(\cdot,\cdot)}$ is  measure on $\cB(V')$.
\end{thm}
\begin{proof} Since  for  Gaussian measures on  finite-dimensional spaces integrals of polynomials are finite, the integrals in (\ref{defLmuscal}) are finite. Hence equation (\ref{defLmuscal}) gives a well-defined  moment  functional $L$ on $A$.

Let $\tau$ be representing cylinder measure for the  moment functional $L$. Then for  each finite-dimensional subspace $\cF$ of $E$, $\tau_\cF$ and $\mu_\cF$ have the same moment functional on  $E'/ \cF^\circ$ and so  have their push forwards on $\cF.$  But $\nu_\cF$ (as defined above) is a Gaussian measure on $\R^d$. It it well-known and follows (for instance) from \cite[Corollary 14.24]{schm2017} that $\nu_\cF$ is determinate. Therefore, both push forwards coincide on $\cF$. Hence $\tau_\cF=\mu_\cF$ which implies that  $\tau=\mu_{(\cdot,\cdot)}$. This shows that  $L$ has a unique representing cylinder measure $\mu_{(\cdot,\cdot)}$, that is, $L$ is determinate. 

If the space $E$ is nuclear,  the cylinder measure $\mu_{(\cdot,\cdot)}$ is indeed a measure  by Minlos' Theorem   \ref{minlosthm}.
\end{proof}

Further, the Gaussian cylinder measure $\mu_{(\cdot,\cdot)}$ yields a positive definite functions $\widehat{\mu_{(\cdot,\cdot)}}$ on $V$, called the Fourier transform of $\mu_{(\cdot,\cdot)}$, defined by
 \begin{align}
\widehat{\mu_{(\cdot,\cdot)}}(y)=\int_{{V'}} e^{{\sf i} \, (\chi,y)} d\mu_{(\cdot,\cdot)}(\chi),\quad  y\in V.
\end{align}
Cylinder measures are often studied in terms of their Fourier transforms and there is a Bochner-Minlos theorem, see e.g. \cite[Chapter IV, \S 4]{gw4}.

 Now we specialize to the case where $V$ is a real separable Hilbert space with scalar product $\langle \cdot,\cdot\rangle$ and assume that the scalar product $(\cdot,\cdot)$  is continuous on $V$. From the latter it follows that there exists a bounded positive self-adjoint operator $b$ on $V$ such that $\ker b=\{0\}$ and 
 \begin{align}\label{bscalarprod}
 (x,y)=\langle bx,y\rangle, \quad 
 x,y \in V.
 \end{align} 
 Conversely, each positive bounded operator $b$ on $V$ with trivial kernel defines a continuous scalar product on $E$ by (\ref{bscalarprod}).
 
 The following result characterizes when the cylinder measure $\mu_{(\cdot,\cdot)}$ is a measure.
 \begin{prop}\label{mstheorem}
 The Gaussian cylinder measure $\mu_{(\cdot,\cdot)}$  on $V'$ defined above yields a measure on the $\sigma$-algebra $\cB(V')$ if and only if the operator $b$ is trace class.
 \end{prop} 
 \begin{proof}
 \cite[Hilfssatz 2, p. 313]{gw4} or \cite[Chapter II, Theorem 2.1]{df}.
 \end{proof}
 \begin{exa}
 Suppose that $\{e_k; k\in \N\}$ is an orthornormal basis of a Hilbert space $V$ and let\, $(b_k)_{k\in \N}$ a bounded sequence of positive  numbers. There is a unique   bounded  operator $b$ on $V$ defined by\,  $b e_k=b_ke_k$ for $k\in \N$. Obviously, $b$ is positive, self-adjoint and  $\ker b=\{0\}$. Let $(\cdot,\cdot)$ denote the scalar product  (\ref{bscalarprod}) on $V$ and $\mu_{(\cdot,\cdot)}$ the corresponding Gaussian cylinder measure on $V'\cong V$. By Theorem \ref{gaussmf},  $\mu_{(\cdot,\cdot)}$ is the {\it unique} representing cylinder measure for the moment functional $L$ (defined by (\ref{defLmuscal})) on $S(V)$.
 
 Clearly, $b$ is a trace class operator if and only if 
 \begin{align}\label{bkinl1}
 \sum_{k=1}^\infty b_k<+\infty.
\end{align} 
Therefore, by Proposition \ref{mstheorem}, the moment functional $L$ has a representing ($\sigma$-additive!) measure on $\cB(V')$ if and only if (\ref{bkinl1}) holds. Thus, if\, $\sum_k b_k=+\infty$, then the unique representing cylinder measure of $L$ is  not a measure.
\end{exa}

\end{document}